\documentclass{elsarticle}
\usepackage{amsthm,amsmath,amssymb}
\usepackage{amssymb}
\usepackage{enumerate,mathdots,multicol}

\newtheorem{Theorem}{Theorem}[section]
\newtheorem{Lemma}[Theorem]{Lemma}

\newtheorem{Corollary}[Theorem]{Corollary}

\newcommand{\A}{{\mathcal{A}}}

\newcommand{\C}{{\mathbb C}}

\newcommand{\Ff}{{\mathbb F}}
\newcommand{\N}{\mathbb{N}}

\newcommand{\RR}{\mathbb{R}}
\newcommand{\Z}{\mathbb{Z}}

\newcommand{\mm}{\mathcal{M}}

\newcommand{\s}{\mathcal{H}}

\DeclareMathOperator{\diag}{diag}
\DeclareMathOperator{\rank}{rank}
\DeclareMathOperator{\Syl}{Syl}

\journal{}

\begin{document}

\begin{frontmatter}
\title{Jordan triple product homomorphisms on Hermitian matrices to and from dimension one}

\author{Damjana Kokol Bukov\v sek}
\ead{damjana.kokol.bukovsek@ef.uni-lj.si}
\address{Faculty of Economics, University of Ljubljana, Kardeljeva plo\v s\v cad 17, Ljubljana,
and Institute of Mathematics, Physics and Mechanics, Department of Mathematics, Jadranska~19, Ljubljana, Slovenia.} 

\author{Bla\v z Moj\v skerc}
\ead{blaz.mojskerc@ef.uni-lj.si}
\address{Faculty of Economics, University of Ljubljana, Kardeljeva plo\v s\v cad 17, Ljubljana,
and Institute of Mathematics, Physics and Mechanics, Department of Mathematics, Jadranska~19, Ljubljana, Slovenia.} 

\begin{abstract}
We characterise all Jordan triple product homomorphisms, that is, mappings $\Phi$ satisfying
$$ \Phi(ABA) = \Phi(A)\Phi(B)\Phi(A) $$
from the set of all Hermitian $n \times n$ complex matrices to the field of complex numbers.
Further we characterise all Jordan triple product homomorphisms from the field of complex or real numbers
or the set of all nonnegative real numbers to the set of all Hermitian $n \times n$ complex matrices.
\end{abstract}

\end{frontmatter}

\section{Introduction}

There are a lot of papers on multiplicative or antimultiplicative maps on matrix spaces. 
Reader can find many facts and references in a survey paper \cite{Semrl} by \v Semrl.
The structure of (anti)multiplicative mappings on the algebra $\mm_n(\Ff)$ of $n \times n$ matrices over field $\Ff$ is well
understood \cite{JL}, but less is known about (anti)multiplicative mappings from $\mm_n(\Ff)$ to $\mm_m(\Ff)$ for $m>n$.
Both, multiplicative and antimultiplicative maps satisfy the identity
$$ \Phi(ABA) = \Phi(A)\Phi(B)\Phi(A) $$
for all $A, B \in \mm_n(\Ff)$. We can study maps $\Phi$ that satisfy merely this identity, called {\it Jordan triple product 
homomorphisms} (J.T.P. homomorphisms for short). Such mappings were studied under additional assumption of additivity on quite general
domain of certain rings \cite{Bresar}. Additive unital J.T.P. homomorphisms preserve Jordan product, that is, satisfy the identity 
$\Phi(AB + BA) = \Phi(A)\Phi(B) + \Phi(B)\Phi(A)$ for every $A, B \in \mm_n(\Ff)$. If characteristic of $\Ff$ is not 2, additive preservers 
of Jordan product are J.T.P. homomorphisms. 
In \cite{Kuzma} Kuzma characterized nondegenerate J.T.P. homomorphisms on the set $\mm_n(\Ff)$
for $n \ge 3$, in \cite{Dob1} Dobovi\v sek characterized J.T.P. homomorphisms from $\mm_n(\Ff)$ to $\Ff$, and in \cite{Dob2} he
characterized J.T.P. homomorphisms from $\mm_2(\Ff)$ to $\mm_3(\Ff)$.

\medskip

In this paper we focus on J.T.P. homomorphisms on the set of all Hermitian complex $n \times n$ matrices. Let us denote by $A^*$ the
complex conjugate of the transpose of matrix $A$ and by $\s_n(\C)$ the set of all Hermitian complex $n \times n$ matrices
$$ \s_n(\C) = \{ A \in \mm_n(\C); A = A^*\}. $$
We cannot study multiplicative or antimultiplicative maps on Hermitian matrices, since they are not closed under multiplication. 
But they are closed under J.T.P., so studying J.T.P. homomorphisms on Hermitian matrices makes perfect sense. Characterization 
of J.T.P. homomorphisms on the set of Hermitian matrices may shed a new light on the structure of Hermitian matrices and may be useful 
in the areas where only Hermitian or positive (semi)definite matrices appear, such as some areas of financial mathematics.

\medskip

The paper is organized as follows. In section 2 we prove some common properties of J.T.P. homomorphisms on the set $\s_n(\C)$. In section 3 we 
characterize all J.T.P. homomorphisms from the set $\s_n(\C)$ to the field $\C$. The result is an analogue to the paper \cite{Dob1} in the case of Hermitian matrices. In section 4 we 
characterize all J.T.P. homomorphisms from the field $\C$, the field $\RR$, or the set  $\RR^+\cup\{0\}$ to the set $\s_n(\C)$. This result is an analogue to the paper \cite{ORS} in the case of Hermitian matrices. 
In the forthcoming paper \cite{KBM} we characterize all J.T.P. homomorphisms from the set $\s_2(\C)$ to the set $\s_2(\C)$.

\section{Common properties of J.T.P. homomorphisms}

In this section we present some common properties of J.T.P. homomorphisms on the set $\s_n(\C)$. First we introduce some notation.
By $I$ we denote the identity matrix of an appropriate dimension. By $\det A$ we denote the determinant and by $\rank A$ the 
rank of a matrix $A$. By direct sum $A \oplus B$ we denote block diagonal matrix $\left[
\begin{array}{cc}
A & 0 \\
0 & B
\end{array} \right]$.

\medskip

We begin with a simple lemma.

\begin{Lemma} \label{lemma21}
Let $A\in \s_n(\C)$ be a Hermitian matrix. Then there exists a unitary Hermition matrix $B\in \s_n(\C)$ such that
$A=B (\lambda \oplus C) B$
with $\lambda \in \RR$ and $C\in \s_{n-1}(\C)$.
\end{Lemma}

\begin{proof}
Since $A\in \s_n(\C)$, $\sigma(A) \subseteq \RR$. Take $\lambda \in \sigma(A)$. Choose an eigenvector $v$ for $\lambda$ such that its first component $v_1$ is real. Suppose $v=e_1$. Then $B:=I$. Otherwise, define $c=\frac{e_1-v}{\| e_1-v\|}$. Observe that $c^*c=1$. Define $B:=I-2c c^*$. A simple calculation shows that $B^2=I$ and $B^*=B$. Next calculate the image of $e_1$:
\begin{align*}
	Be_1 &= e_1 -2 c c^* e_1 = e_1- \frac{2}{\|e_1-v\|^2} (e_1-v)(e_1-v)^* e_1 \\
	&= e_1 -\frac{2}{\|e_1-v\|^2} (e_1-v)(1-v^*e_1) = e_1-\frac{2(1-v_1)}{\|e_1-v\|^2} (e_1-v)\\
	&= [1,\, 0,\, \ldots, \, 0]^T -\frac{2-2v_1}{2-2v_1} \left[ 1-v_1, \, -v_2, \,\ldots,\, -v_n\right]^T=v,
\end{align*}
since $\|e_1-v\|^2=-2 \langle e_1, v\rangle +1 +\|v\|=2-2v_1$.

Observe that $$BABe_1=BAv=B\lambda v=\lambda Bv=\lambda e_1,$$
hence the first column of $BAB$ equals $[\lambda, \, 0,\ldots,\, 0]^T$. With $BAB$ symmetric we obtain the required form $BAB=\lambda \oplus C$.
\end{proof}

\begin{Corollary} \label{cor}
Let $A\in \s_n(\C)$ be a Hermitian matrix. Then there exist unitary Hermition matrices $B_1, \ldots, B_{n-1} \in \s_n(\C)$  such that
$$A=B_1 B_2 \ldots B_{n-1} D B_{n-1} \ldots B_1$$
with $D$ diagonal.
\end{Corollary}

\begin{proof}
We use induction on $n$ to get the desired result.
\end{proof}

We use this result to prove a very useful observation.
 
\begin{Lemma} \label{sim}
Let $\Phi: \s_n(\C) \to \s_m(\C)$ be a J.T.P. homomorphism with $\Phi(I) = I$, and $A,C\in \s_n(\C)$ similar matrices. Then $\Phi(A)$ and $\Phi(C)$ are unitarily similar.
\end{Lemma}

\begin{proof}
By Corollary \ref{cor} there exist unitary Hermitian matrices $B_1, \ldots, B_{n-1}$, $B'_1,\ldots,B'_{n-1}$ and diagonal matrix $D$ such that
$$A=B_1 \ldots B_{n-1}D B_{n-1} \ldots B_1$$
and
$$C=B'_1 \ldots B'_{n-1} D B'_{n-1} \ldots B'_1.$$
Then
\begin{align*}
\Phi(A) &= \Phi(B_1) \ldots \Phi(B_{n-1}) \Phi(D) \Phi(B_{n-1})\ldots \Phi(B_1)
\end{align*}
and similarly, 
\begin{align*}
\Phi(C) &= \Phi(B'_1) \ldots \Phi(B'_{n-1}) \Phi(D) \Phi(B'_{n-1})\ldots \Phi(B'_1),
\end{align*}
where $\Phi(B_i)=\Phi(B_i)^*=\Phi(B_i)^{-1}$ and $\Phi(B'_i)=\Phi(B_i)^*=\Phi(B_i)^{-1}$, which gives us unitary similarity of $\Phi(A)$ and $\Phi(C)$.
\end{proof}

\begin{Lemma} \label{potence}
Let $\Phi:\A \to \mathcal{B}$ be a J.T.P. homomorphism with $\Phi(I)=I$, where $\A,\mathcal{B}\in \{\RR, \RR^+, \C, \s_n(\RR), \s_n(\C) \}$. Then $\Phi(A^n)=\Phi(A)^n$ for every $A\in \A$ and every $n\in\N$, and $\Phi(A^n)=\Phi(A)^n$ for every $A\in \A$ invertible and $n\in\Z$.
\end{Lemma}

\begin{proof}
Take $A\in\A$. Then
$$\Phi(A^2)=\Phi(A I A)= \Phi(A) \Phi(I)\Phi(A)=\Phi(A)^2$$
and $\Phi(A^3)=\Phi(A)^3$.

Take $n>3$. Use induction on $n$ to see that if $n=2k+1$ then $$\Phi(A^n)=\Phi(A^k A A^k)=\Phi(A^k) \Phi(A) \Phi(A^k)=\Phi(A)^k \Phi(A) \Phi(A)^k=\Phi(A)^n,$$
or else $n=2k$ and $$\Phi(A^n)=\Phi(A^{2k})=\Phi(A^k)^2=\Phi(A)^n.$$

For the second part of lemma take $A$ invertible. Then $$I=\Phi(I)=\Phi(A^{-1} A^2 A^{-1})=\Phi(A^{-1}) \Phi(A)^2 \Phi(A^{-1}),$$
hence $\Phi(A)$ invertible. Now take
$$\Phi(A)=\Phi(A A^{-1} A)=\Phi(A) \Phi(A^{-1}) \Phi(A).$$
Since $\Phi(A)$ is invertible, we get $\Phi(A)\Phi(A^{-1})=I$ which gives us $\Phi(A^{-1})=\Phi(A)^{-1}$.
\end{proof}

\section{J.T.P. homomorphisms from $\s_n(\C)$ to $\C$}

In this section we study J.T.P. homomorphisms that map from $n \times n$ Hermitian matrices to the complex field.

\medskip

Let a mapping $\Phi: \s_n(\C) \to \C$ be a J.T.P. homomorphism. Since $\Phi(I)^3=\Phi(I)$ and $\Phi(0)^3=\Phi(0)$, we get the following cases.

\medskip

\noindent
{\sc Case 1:} $\Phi(I)=0$. Using $\Phi(I \cdot A \cdot I) = \Phi(I)\Phi(A)\Phi(I)$ for $A$ arbitrary we see that $\Phi \equiv 0$.

\medskip

\noindent
{\sc Case 2:} $\Phi(I)=-1$. Define a map $\Phi': \s_n(\C) \to \C$ with $\Phi'(A)=-\Phi(A)$ for $A\in \s_n(\C)$. Then $\Phi'(I)=1$ which translates to Cases 3 and 4 for $\Phi'$.

\medskip

Last option is $\Phi(I)=1$. Then by Lemma \ref{potence} we have $\Phi(0)^2=\Phi(0)$, hence we either get $\Phi(0)=0$ or $\Phi(0)=1$.

\medskip

\noindent
{\sc Case 3:} $\Phi(I)=1$ and $\Phi(0)=1$. Then $\Phi \equiv 1$ since $\Phi(0 \cdot A \cdot 0) = \Phi(0)\Phi(A)\Phi(0)$ for $A$ arbitrary.

\medskip

\noindent
{\sc Case 4:} $\Phi(I)=1$ and $\Phi(0)=0$.

\medskip

\noindent Until further notice we assume Case 4.

\begin{Lemma} \label{rank-less-n}
Let $\Phi: \s_n(\C) \to \C$ be a J.T.P. homomorphism with $\Phi(I)=1$ and $\Phi(0)=0$. Then $\Phi(A)=0$ for every $A\in \s_n(\C)$ with $\rank A<n$.
\end{Lemma}

\begin{proof}
Define $$k=\min_{A\in \s_n} \{ \rank A: \, \Phi(A)\neq 0 \}.$$
Suppose there exists $A\in \s_n(\C)$ with $\rank A=k<n$ such that $\Phi(A) \neq 0$. By Lemma \ref{sim} similar matrices are mapped to the same complex number, so we may assume that matrix $A$ is diagonal, hence $A=\diag(\lambda_1, \ldots,\lambda_k,0,\ldots,0)$ with $\lambda_1,\ldots,\lambda_k \in \RR\setminus \{0\}$.
Define
$$E=\left[
\begin{array}{cccc}
0 & & 1 \\
 &  \iddots & \\
1 &   &0
\end{array} \right].$$
It holds that $E^2=I$, hence $\Phi(E)^2=1$. Further, let
$$B=EAE=\diag(0,\ldots,0, \lambda_k,\ldots,\lambda_1) .$$ 
We have $\Phi(B)=\Phi(A)$, so 
$$\Phi(A^3)=\Phi(A)^3=\Phi(A)\Phi(B) \Phi(A)= \Phi(ABA)=0,$$
since $\rank ABA <k$. Hence $\Phi(A)=0$, which is a contradiction.
\end{proof} 

We can now characterise unital J.T.P. homomorphisms on positive definite matrices.  

\begin{Lemma} \label{pos}
Let $\Phi: \s_n(\C) \to \C$ be a J.T.P. homomorphism with $\Phi(I)=1$ and $\Phi(0)=0$. Then there exists some multiplicative map $\Psi: (0,\infty) \to \C$ such that $\Phi(A)=\Psi(\det A)$ for every positive definite matrix $A$.
\end{Lemma}

\begin{proof}
A matrix $A$ is positive definite, hence $\sigma(A) \subseteq (0,\infty)$. Define $$\Psi(a)=\Phi(\diag(a,1,\ldots,1)).$$
Then it holds that
\begin{align*}
	\Phi(A) &= \Phi( \diag(\lambda_1,\ldots,\lambda_n)) \\
		&= \Phi( \diag(\sqrt{\lambda_1},1,\ldots,1)) \Phi( \diag(1,\lambda_2,\ldots, \lambda_n)) \Phi(\diag(\sqrt{\lambda_1},1,\ldots,1)) \\
		&= \Psi(\sqrt{\lambda_1}) \Phi(\diag(1,\lambda_2,\ldots,\lambda_n)) \Psi(\sqrt{\lambda_1}) \\
		&= \Psi(\lambda_1) \Phi(\diag(1,\lambda_2,\ldots,\lambda_n)) \\
		&= \Psi(\lambda_1) \cdots \Psi(\lambda_n) \Phi(I)= \Psi(\lambda_1 \cdots \lambda_n).
\end{align*}
It also holds that
\begin{align*}
	\Psi(ab) &= \Phi(\diag(\sqrt{a},1,\ldots,1)) \Phi(\diag(b,1,\ldots,1)) \Phi(\diag(\sqrt{a},1,\ldots,1)) \\
	&= \Psi(\sqrt{a}) \Psi(b) \Psi(\sqrt{a})= \Psi(a) \Psi(b),
\end{align*}
which concludes the proof.
\end{proof}

Now take some $A\in \s_n(\C)$ invertible. Then $A$ is similar to $\diag(\lambda_1,\ldots,\lambda_n)$ with $\lambda_1,\ldots,\lambda_k>0$ and $\lambda_{k+1},\ldots,\lambda_n<0$, so we may assume that $A$ is diagonal by Lemma \ref{sim}. We can decompose $A$ into the following triple product:
$$A=\diag(\sqrt{|\lambda_1|},\ldots,\sqrt{|\lambda_n|}) \diag(1,\ldots,1,-1,\ldots,-1) \diag(\sqrt{|\lambda_1|},\ldots,\sqrt{|\lambda_n|}).$$
Denote with $\Syl(A)$ the {\it inertia} of a matrix $A$, that is, the number of positive eigenvalues of a matrix $A$.
Using Lemma \ref{pos} we see that
$$\Phi(A)=\Psi(|\det A|) \Phi(\diag(1,\ldots,1,-1,\ldots,-1)).$$
Then there exists some mapping $\eta: \{0,1,\ldots,n\} \to \{-1,1\}$ such that
$$\eta(k)=\Phi(\diag(\underbrace{1,\ldots,1}_k,-1,\ldots,-1)).$$
Hence
\begin{equation} \label{char}
\Phi(A)= \Psi(|\det A|) \eta(\Syl(A)).
\end{equation}

Since we are assuming $\Phi(I)=1$, we get $\eta(n) = 1$. In the case $\Phi(I)=-1$, we get 
$\Phi(A)=-\Phi'(A)=-\Psi(|\det A|) \eta(\Syl(A)) = \Psi(|\det A|) \eta'(\Syl(A))$, where now  $\eta'(n) = -1$.
Because $\Phi(A)=0$ for every $A\in \s_n(\C)$ with $\rank A<n$, equality \eqref{char} holds also for noninvertible matrices.
Cases 1 and 3 can also be written in this form.

\medskip

To see that the converse also holds, we have to prove that every mapping of the form \eqref{char} is also a J.T.P. homomorphism. Take an arbitraty $\Psi:(0,\infty) \to \C$ a multiplicative mapping and $\eta$ some mapping from $\{0,1,\ldots,n\}$ to $\{-1,1\}$.
Then $$\Phi(ABA)=\Psi(|\det(ABA)|) \eta(\Syl(ABA)).$$
On the other hand,
\begin{align*}
\Phi(A) \Phi(B) &\Phi(A)= \\ &= \Psi(|\det A|) \Psi(|\det B|) \Psi(|\det A|) \eta(\Syl(A)) \eta(\Syl(B)) \eta(\Syl(A)) \\
&= \Psi(|\det(ABA)|) \eta(\Syl(B))
\end{align*}
If $\det A=0$, the J.T.P. property obviously applies, so we need to consider only the case when $A$ is invertible. But then the matrices $B$ and $ABA$ are congruent since $A$ is Hermitian, hence we can apply the famous Sylvester's Law of Inertia for Hermitian matrices which implies that matrices $B$ and $ABA$ have the same inertia, i.e., $\Syl(B)=\Syl(ABA)$.

\medskip

We can summarize our findings in the following theorem, which is the main result of this section.

\begin{Theorem} \label{thm3.3}
Let $\Phi$ be a mapping from $\s_n(\C)$ to $\C$. Then $\Phi$ satisfies the identity $\Phi(ABA)=\Phi(A) \Phi(B) \Phi(A)$ 
if and only if $\Phi$ has the form
$$\Phi(A)=\Psi(|\det A|) \eta(\Syl(A)),$$ 
where $\Psi: [0,\infty) \to \C$ is a multiplicative function, $\eta:\{0,1,\ldots,n\} \to \{-1,1\}$ an arbitrary mapping, and $\Syl(A)$ the inertia of $A$.
\end{Theorem}

\section{J.T.P. homomorphisms from $\C$ to $\s_n(\C)$ }

In this section we study J.T.P. homomorphisms from one-dimensional matrices to $n\times n$
Hermitian matrices. First we look at mappings from non-zero complex numbers $\C^*$, 
non-zero real numbers $\RR^*$, or positive real numbers $\RR^+$ to invertible $n\times n$
Hermitian real or complex matrices.

\begin{Lemma} \label{C_to_invertible}
Let a mapping $\Phi: \A \to \s_n(\C)$ be a J.T.P. homomorphism where $\A$ is the set $\C^*$, $\RR^*$, or $\RR^+$, such that 
$\Phi(\lambda)$ is invertible for every $\lambda \in \A$ and $\Phi(1) = I$. Then there exist a unitary matrix $T$ and 
multiplicative maps $\varphi_1, \varphi_2, ..., \varphi_n : \A \to \RR^*$ with $\varphi_i(1) = 1$, such that
$$\Phi(\lambda) = T \left[
\begin{array}{cccc}
\varphi_1(\lambda) & 0 & \cdots & 0 \\
0 & \varphi_2(\lambda) & \cdots & 0 \\
\vdots & \vdots & \ddots & \vdots \\
0 & 0 & \cdots & \varphi_n(\lambda) 
\end{array} \right] T^*, \ \ \lambda \in \A.$$
\end{Lemma}

\begin{proof}
First we prove the lemma for the cases $\A = \C^*$ and $\A = \RR^+$. We use the induction on $n$. 
For $n=1$ the statement is obvious, since a J.T.P. homomorphism
on commutative sets is a multiplicative map. So, assume that $n > 1$ and the lemma is true for all positive 
integers smaller than $n$. If for every $\lambda \in \A$ matrix $\Phi(\lambda)$ has only one eigenvalue, then 
every $\Phi(\lambda)$ is a scalar matrix and it has asserted form. So, let $\lambda_0 \in \A$ be such a number
that $\Phi(\lambda_0)$ has at least two distinct eigenvalues. Since $\lambda_0$ has a square root in $\A$, we have 
$\Phi(\lambda_0) = \Phi(\sqrt{\lambda_0})^2$ by Lemma \ref{potence}, hence $\Phi(\lambda_0)$ is positive definite.
 Let $\mu_1 > \mu_2 > ... > \mu_k > 0$ be ordered distinct eigenvalues 
of $\Phi(\lambda_0)$. Then $\Phi(\lambda_0)$ can be written as
$$\Phi(\lambda_0) = T \left[
\begin{array}{cccc}
\mu_1 I & 0 & \cdots & 0 \\
0 & \mu_2 I & \cdots & 0 \\
\vdots & \vdots & \ddots & \vdots \\
0 & 0 & \cdots & \mu_k I
\end{array} \right] T^*.$$ 
Using $\Phi'(\lambda) = T^* \Phi(\lambda) T$ instead of $\Phi(\lambda)$, we may assume that $\Phi(\lambda_0) = A$ is diagonal 
matrix with decreasing diagonal. We are going to prove that $\Phi(\lambda)$ is block diagonal with respect to above decomposition
for every $\lambda \in \A$. So let $\lambda \in \A$ be arbitrary. Then
$$\Phi(\lambda) = B = \left[
\begin{array}{cccc}
X_{1} & Y_{12} & \cdots & Y_{1k} \\
Y_{21} & X_{2} & \cdots & Y_{2k} \\
\vdots & \vdots & \ddots & \vdots \\
Y_{k1} & Y_{k2} & \cdots & X_{k}
\end{array} \right] .$$ 
Since $B$ is Hermitian, we have $X_i$ is Hermitian for every $i$ and $Y_{ij} = Y_{ji}^*$ for every $i \ne j$.
Now we use J.T.P. homomorphism property 
$$ AB^2A = \Phi(\lambda_0)\Phi(\lambda)^2\Phi(\lambda_0) = \Phi(\lambda_0\lambda^2\lambda_0) =$$
$$= \Phi(\lambda\lambda_0^2\lambda) = \Phi(\lambda)\Phi(\lambda_0)^2\Phi(\lambda) = BA^2B .$$
Diagonal blocks of matrix $AB^2A$ are equal to
$$ \mu_i^2(X_i^2 + \sum_{j\ne i}Y_{ij}Y_{ji}) = \mu_i^2X_i^2 + \sum_{j\ne i}\mu_i^2Y_{ij}Y_{ij}^*$$
and diagonal blocks of matrix $BA^2B$ are
$$ \mu_i^2X_i^2 + \sum_{j\ne i}\mu_j^2Y_{ij}Y_{ij}^* .$$
So we get
$$\sum_{j\ne i}(\mu_i^2 - \mu_j^2)Y_{ij}Y_{ij}^* = 0$$
for every $i=1, 2, ... k.$ For $i=1$ this means
$$\sum_{j = 2}^k(\mu_1^2 - \mu_j^2)Y_{1j}Y_{1j}^* = 0.$$
Since $\mu_1$ is the largest eigenvalue of $A$, we have a positive linear combination of positive semidefinite
matrices $Y_{1j}Y_{1j}^*$, which is equal to 0, so $Y_{1j}Y_{1j}^* = 0$ and thus $Y_{1j} = 0$ for every $j = 2, ..., k.$
All off-diagonal blocks of $B$ in the first row and column are zero. Using induction and similar argument we prove
that all off-diagonal blocks of $B$ are zero, so $B$ is block diagonal. We now have that 
$$\Phi(\lambda) = B = \left[
\begin{array}{cccc}
\rho_1(\lambda) & 0 & \cdots & 0 \\
0 & \rho_2(\lambda) & \cdots & 0 \\
\vdots & \vdots & \ddots & \vdots \\
0 & 0 & \cdots & \rho_k(\lambda)
\end{array} \right] .$$ 
Since $\Phi$ is J.T.P. homomorphism, also $\rho_i$ are J.T.P. homomorphisms for every $i$ and they map to matrices of size
smaller than $n$. Hence they have the asserted form and the proof is finished for the cases $\A = \C^*$ and $\A = \RR^+$.

In the case $\A = \RR^*$ we can proceed the same way as in the other two cases, if there exists $\lambda_0 > 0$, 
such that $\Phi(\lambda_0)$  has at least two distinct eigenvalues. If this is not true, we may assume that $\Phi(\lambda)$ has only one eigenvalue for every $\lambda > 0$
and that there exists $\lambda_0 < 0$ such that $\Phi(\lambda_0)$  has at least two distinct eigenvalues. Then there exists
a multiplicative map $\varphi: \RR^+ \to \RR^*$, such that $\Phi(\lambda) = \varphi(\lambda) I$ for every $\lambda > 0$.
Since $\lambda_0^2 > 0$,
$\Phi(\lambda_0)^2$ has only one eigenvalue, and up to unitary similarity $\Phi(\lambda_0)$ has the form
$$\Phi(\lambda_0) = \left[
\begin{array}{cc}
\mu I_k & 0 \\
0 & -\mu I_{n-k}
\end{array} \right] .$$
Now, for any $\lambda < 0$, we have $\lambda/\lambda_0 > 0$, so 
$$\Phi(\lambda) = 
\Phi(\sqrt{\lambda/\lambda_0} \lambda_0 \sqrt{\lambda/\lambda_0}) = 
\Phi(\sqrt{\lambda/\lambda_0}) \Phi(\lambda_0) \Phi(\sqrt{\lambda/\lambda_0}) = $$
$$ = \left[
\begin{array}{cc}
\varphi(\lambda/\lambda_0)\mu I_k & 0 \\
0 & -\varphi(\lambda/\lambda_0)\mu I_{n-k}
\end{array} \right] .$$
Thus, for every $\lambda \in\RR^*$ $\Phi(\lambda)$ has the asserted form and the proof is complete.
\end{proof}

Now we can look at mappings from complex numbers $\C^*$, 
real numbers $\RR^*$, or nonnegative real numbers $\RR^+\cup\{0\}$ to the set of all $n\times n$
Hermitian real or complex matrices. We get the following theorem.

\begin{Theorem} \label{C_to_all}
Let a mapping $\Phi: \A \to \s_n(\C)$ be a J.T.P. homomorphism where $\A$ is the set $\C$, $\RR$, or $\RR^+\cup\{0\}$.
Then there exist a unitary matrix $T$, a diagonal matrix $D$ with $\pm 1$'s on its diagonal and 
multiplicative maps $\varphi_1, \varphi_2, ..., \varphi_n : \A \to \RR$, such that
$$\Phi(\lambda) = T D\left[
\begin{array}{cccc}
\varphi_1(\lambda) & 0 & \cdots & 0 \\
0 & \varphi_2(\lambda) & \cdots & 0 \\
\vdots & \vdots & \ddots & \vdots \\
0 & 0 & \cdots & \varphi_n(\lambda) 
\end{array} \right] T^*, \ \ \lambda \in \A.$$
\end{Theorem}

\begin{proof}
First notice that $\Phi(1)$ is a tripotent, so its eigenvalues are equal to 0, 1, or $-1$.
Up to unitary similarity we have
$$\Phi(1) = \left[
\begin{array}{ccc}
I_k & 0 & 0 \\
0 & -I_s & 0 \\
0 & 0 & 0_{n-k-s}
\end{array} \right] .$$
With respect to this decomposition we have for every $\lambda \in \A$
$$\Phi(\lambda) = A = \left[
\begin{array}{ccc}
A_{11} & A_{12} & A_{13} \\
A_{21} & A_{22} & A_{23} \\
A_{31} & A_{32} & A_{33}
\end{array} \right] .$$
Since $A = \Phi(\lambda) = \Phi(1 \cdot \lambda \cdot 1) = \Phi(1)\Phi(\lambda)\Phi(1)$, we obtain that $A_{12} = A_{13} = A_{21} = 
A_{23} = A_{31} = A_{32} = A_{33} = 0$. Let us take
$$D = \left[
\begin{array}{ccc}
I_k & 0 & 0 \\
0 & -I_s & 0 \\
0 & 0 & I_{n-k-s}
\end{array} \right] .$$
We have
$$\Phi(\lambda) = D \left[
\begin{array}{ccc}
\rho_1(\lambda) & 0 & 0 \\
0 & \rho_2(\lambda) & 0 \\
0 & 0 & 0
\end{array} \right] ,$$
where $\rho_1: \A \to \s_k(\C)$, $\rho_2: \A \to \s_s(\C)$ are J.T.P. homomorphisms with $\rho_1(1) = I_k$ and $\rho_2(1) = I_s$.
We take $\varphi_{k+s+1}, ..., \varphi_n : \A \to \RR$ multiplicative maps identically equal to zero.
To prove the theorem, we may add additional assumption that $\Phi(1) = I$.
Now, since $\Phi(0) = \Phi(0 \cdot 1 \cdot 0) = \Phi(0)^2$, we have up to similarity that
$$\Phi(0) = \left[
\begin{array}{cc}
I_k & 0 \\
0 & 0_{n-k}
\end{array} \right] .$$
With respect to this decomposition we have for every $\lambda \in \A$
$$\Phi(\lambda) = A = \left[
\begin{array}{cc}
A_{11} & A_{12}  \\
A_{21} & A_{22}  
\end{array} \right] .$$
Since $\Phi(0) = \Phi(0 \cdot \lambda \cdot 0) = \Phi(0)\Phi(\lambda)\Phi(0)$, we obtain that $A_{11} = I$. 
Further $\Phi(0) = \Phi(\lambda \cdot 0 \cdot \lambda) = \Phi(\lambda)\Phi(0)\Phi(\lambda)$, so $A_{12} = A_{21} = 0$. 
We have
$$\Phi(\lambda) = \left[
\begin{array}{cc}
I_k & 0  \\
0 & \rho_3(\lambda) 
\end{array} \right] .$$
$\rho_3: \A \to \s_{n-k}(\C)$ is a J.T.P. homomorphism with $\rho_3(1) = I_{n-k}$ and $\rho_3(0) = 0$.
We take $\varphi_{1}, ..., \varphi_k : \A \to \RR$ multiplicative maps identically equal to 1.
For any non-zero $\lambda \in \A$ we have $ I = \rho_3(1) = \rho_3(\lambda\lambda^{-2}\lambda) = 
\rho_3(\lambda)\rho_3(\lambda^{-2})\rho_3(\lambda)$, so $\rho_3(\lambda)$ is an invertible matrix.
By Lemma \ref{C_to_invertible} it has diagonal form up to unitary similarity. Obtained multiplicative maps 
$\varphi_i : \A\setminus\{0\} \to \RR$ can be extended to the whole set $\A$, since $\rho_3(0) = 0$.
\end{proof}

\bigskip

\bigskip

\end{document}